\documentclass[11pt]{amsart}
\newtheorem{theorem}{Theorem}[section]
\newtheorem{lemma}[theorem]{Lemma}
\newtheorem{proposition}[theorem]{Proposition}
\theoremstyle{definition}
\newtheorem{definition}[theorem]{Definition}

\newtheorem{corollary}[theorem]{Corollary}
\theoremstyle{remark}
\newtheorem{remark}[theorem]{Remark}

\begin{document}
\title{Laplacian Flow for Closed $G_2$-Structures: Short Time Behavior}
\author{Robert Bryant \& Feng Xu}
\address{MSRI, 17 Gauss Way, Berkeley, CA 94720 \& MSI, The Australian National University, ACT 0200, Australia}
\email{bryant@msri.org, feng.xu@anu.edu.au}
\date{}
\begin{abstract}
We prove short time existence and uniqueness of solutions to the Laplacian flow for closed $G_2$ structures on a compact manifold $M^7$. The result was claimed in \cite{BryantG2}, but its proof has never appeared. 
\end{abstract}
\maketitle
\section*{Introduction}
Since R. Hamilton introduced Ricci flow to study Riemannian structures on manifolds \cite{Hamilton}, extensive work has been done to extend the ideas and techniques to other geometric structures. For example, in the category of K\"ahler structures, K\"ahler Ricci flow was studied in \cite{Cao}. In the symplectic background, a flow called anti-complexified Ricci flow stimulates much interest and its s hort time existence has been proved in \cite{LW}. 

 Many attempts have also been tried on $G_2$ structures (e.g., see \cite{Ka},\cite{KY}). In this article, we are interested in the Laplacian flow
\begin{eqnarray}\label{LapFlow}
 \frac{d}{dt}\sigma=\Delta_\sigma\sigma,
\end{eqnarray}
where $\sigma$ is the defining 3-form of a $G_2$-structure, and $\Delta_\sigma$ is the Hodge Laplacian of the metric determined by $\sigma$. This flow was studied by Steven Altschuler and the first author between 1992 and 1994 (see \cite{BryantG2}). In particular, when the initial 3-form $\sigma_0$ is closed and we evolve inside a fixed cohomology class, a natural question is under which conditions it will converge to a structure with holonomy in $G_2$. Later on, when $M$ is compact, it is found to be the gradient flow of Hitchin's volume functional $V:[\sigma_0]\rightarrow \mathbf{R} $ defined by
\begin{eqnarray}\label{Functional}
 V(\sigma)=\int_M\sigma\wedge *_\sigma\sigma.
\end{eqnarray}

However, since both the flow (\ref{LapFlow}) and the functional (\ref{Functional}) are diffeomorphism invariant and since $[\sigma_0]/{\rm Diff}_0$ is still infinite dimensional, it is not clear if a short time solution to 
\begin{eqnarray}\label{ClsdLapFlow}
\left\{\begin{array}{l}
 \frac{d}{dt}\sigma=\Delta_\sigma\sigma\\
  d\sigma=0\\
  \sigma(0)=\sigma_0
\end{array}\right.
\end{eqnarray}
exists at all. We prove this result in this paper.
\begin{theorem}\label{Thm}
 Assume $M$ is compact. Then the initial value problem (\ref{ClsdLapFlow}) has a unique solution for a short time $0\leq t\leq \epsilon$ with $\epsilon$ depending on $\sigma_0$. 
\end{theorem}

We give an outline of its proof. First, when restricted to closed forms, the flow equation takes the form
 \begin{eqnarray}
  \left\{\begin{array}{l}
   \frac{d}{dt}\sigma=-d*_\sigma d*_\sigma\sigma\\\\
   \sigma(0)=\sigma_0    
  \end{array}\right.
 \end{eqnarray}
where $d\sigma_0=0.$

 It follows that a solution $\sigma(t)$ must lie in the cohomology class $[\sigma_0]$. By letting $\sigma(t)=\sigma_0+\theta(t)$ with $\theta$ taking values in exact forms, we may rewrite the flow equation in terms of $\theta$
\begin{eqnarray}\label{Flowtheta}
 \left\{\begin{array}{l}
          \frac{d}{dt}\theta=-d*_{(\theta+\sigma_0)}d*_{(\theta+\sigma_0)}(\theta+\sigma_0),\\\\
          \theta(0)=0.
        \end{array}\right.
\end{eqnarray}
Clearly, (\ref{ClsdLapFlow}) is equivalent to the initial value problem (\ref{Flowtheta}) for a family of exact forms $\theta(t)$.

Of course, this flow is still diffeomorphism-invariant and thus is not parabolic. As in Deturck's Trick for the Ricci Flow \cite{De}, we modify the flow by an operator of the form $\mathcal{L}_{V(\sigma)}(\sigma)=d(V\lrcorner \sigma)+V\lrcorner d\sigma=d(V\lrcorner \sigma)$:
 \begin{eqnarray} \label{ModFlowtheta}
  \left\{\begin{array}{l}
   \frac{d}{dt}\theta=-d*_{(\sigma_0+\theta)}d*_{(\sigma_0+\theta)}(\sigma_0+\theta)+d(V(\theta+\sigma_0)\lrcorner(\theta+\sigma_0)),\\\\
   \theta_0=0.
  \end{array}\right.
 \end{eqnarray}

 For wisely chosen $V$, this new flow is elliptic in the direction of closed forms. We hope to show that the new flow has short time existence and then, by applying suitable time-dependent diffeomorphisms to get a solution of the flow (\ref{Flowtheta}).

 However, no existing theory of parabolic equations seems to be  directly applicable. Thus, it seems to us that an argument along the line of inverse function theorems is necessary. In fact, we will use Nash Moser inverse function theorem for tame F\'echet spaces (see \cite{Hamilton1}, pp. 171-172). We introduce 
\begin{eqnarray}\label{OpenU}\mathcal{U}=\left\{\theta\in d C^\infty\left([0,T]\times M, \Lambda^2(M)\right):\sigma_0+\theta \quad {\rm is \quad definite}\right\}.\end{eqnarray} 
This is an open set of the Fr\'echet space $dC^\infty\left([0,T]\times M ,\Lambda^2(M)\right)$. We consider a mapping $F:\mathcal{U}\rightarrow C^\infty\left([0,T]\times M, \Lambda^2(M)\right)\times dC^\infty\left(M,\Lambda^2(M)\right)$ defined by
\begin{eqnarray}\label{F}\theta\mapsto \left(\frac{d}{dt}\theta-\Delta_\sigma\sigma-\mathcal{L}_{V(\sigma)}\sigma,\theta|_{t=0}\right).\end{eqnarray}
where $\sigma=\sigma_0+\theta$ and $V(\sigma)$ is a vector field dependent on 1st order derivatives of $\sigma$. We will prove the following lemma:
\begin{lemma} \label{Lemma}
 Suppose $\theta\in \mathcal{U}$ is a solution to 
\begin{eqnarray}\label{ModFlow}
 \left\{\begin{array}{l}
 \frac{d}{dt}\theta-\Delta_\sigma\sigma-\mathcal{L}_{V(\sigma)}\sigma=\chi\\\\
  \theta(0)=\theta_0.
\end{array}\right.
\end{eqnarray}
Then for $(\overline{\chi},\overline{\theta}_0)$ sufficiently close to $(\chi,\theta_0)$, there is a unique solution $\overline{\theta}(t)$ to 
\begin{eqnarray*}
 \left\{\begin{array}{l}
 \frac{d}{dt}\overline{\theta}-\Delta_{\overline{\sigma}}\overline{\sigma}-\mathcal{L}_{V(\overline{\sigma})}\overline{\sigma}=\overline{\chi}\\\\
  \overline{\theta}(0)=\overline{\theta}_0.
\end{array}\right.
\end{eqnarray*}
\end{lemma}
 
It is for the proof of Lemma \ref{Lemma} that we use Nash Moser inverse function theorem. To use the theorem, we need to show:
\begin{enumerate}
 \item The linearized map $F_*$ at $\sigma$ is injective;
 \item $F_*$ is surjective;
 \item The inverse map $F_*^{-1}$ is a  {\it smooth tame} map.                                        
\end{enumerate}
Roughly speaking, (3) means that the solutions to linearized equations satisfy certain a priori estimates. 

Below is the structure of the rest of this paper. In \textsection \ref{G2str}, we review $G_2$ structures. We discuss $G_2$ connections and the Levi-Civita connections. These will be useful in constructing vector fields used in Deturck's Trick. Most importantly, differential identities of $G_2$ holonomy are discussed. These identities are discovered in \cite{BryantG2} and are basis for our later computations. Hitchin's functional will also be discussed in more detail. In \textsection \ref{Proof}, we give a detailed proof of Theorem \ref{Thm} assuming Lemma \ref{Lemma}. In \textsection\ref{ProofLemma}, we prove Lemma \ref{Lemma}. \\

The work was done when the second named author was an MSRI postdoctoral fellow. He would like to thank MSRI for hospitality.

\section{Manifolds with $G_2$-structures}\label{G2str}

 In this section, we collect basic facts about $G_2$-structures on a 7-dimensional manifold $M$. For details, see \cite{BryantG2}.

\subsection{Lie group $G_2$}
Let $V=\mathbf{R}^7$ be the 7-dimensional Euclidean space. Let $\{e_1,e_2,\cdots,e_7\}$ be the standard basis of $V$ and let $\{e^1,\cdots,e^7\}$ be the dual basis. 

Define a 3-form 
\begin{eqnarray}
 \phi=e^{123}+e^{145}+e^{167}+e^{246}-e^{257}-e^{347}-e^{256}
\end{eqnarray}
 where $e^{123}=e^1\wedge e^2\wedge e^3$, etc.

We have the following characterization of $G_2$:
\begin{lemma}[\cite{Schouten},\cite{BryantHolonomy}]
 \begin{eqnarray}G_2=\{g\in GL(V)|g^*\phi=\phi\}.\end{eqnarray}
\end{lemma}
The Lie group $G_2$ is compact, connected, simply connected and simple. It acts irreducibly on $V$ and preserves the metric and orientation for which $\{e_1, \cdots, e_7\}$ is an oriented orthonormal basis. We use $*_\phi$ (or $*$ when there is no danger of confusion) to denote the Hodge star operator determined by the orientation and the metric. It particular, $G_2$ also preserves the $4$-form
\begin{eqnarray}
 *_\phi\phi=e^{4567}+e^{2367}+e^{2345}+e^{1357}-e^{1346}-e^{1256}-e^{1247}.
\end{eqnarray}

  Using the metric, we identify $V$ and its dual $V^*$ through musical isomorphisms 
\[^\flat:V\rightarrow V^*\]
and 
\[_\#:V^*\rightarrow V.\]

Under the action of $G_2$, the spaces $\Lambda^1V^*$ and $\Lambda^6V^*$ are irreducible while the other wedge products $\Lambda^*(V^*)$ decomposes into irreducible components:
\begin{eqnarray}
 \Lambda^2V^*=\Lambda^2_7\oplus \Lambda^2_{14}\\
 \Lambda^3V^*=\Lambda^3_1\oplus \Lambda^3_7\oplus \Lambda^3_{27}\\
\Lambda^4V^*=\Lambda^4_1\oplus \Lambda^4_7\oplus \Lambda^4_{27}\\
 \Lambda^5V^*=\Lambda^5_7\oplus \Lambda^5_{14}.
\end{eqnarray}
Here $\Lambda^p_d$ denotes an irreducible $G_2$ representation of dimension $d$. Different spaces in this list with the same dimension are naturally isomorphic. For example, $*:\Lambda^4_{27}\simeq\Lambda^3_{27}$. The space $\Lambda^2_{14}$ is naturally isomorphic to the Lie algebra $\mathfrak{g}_2$ through musical isomorphisms.

The space $\Lambda^3_{27}$ deserves special attention. An explicit algebraic description of this space is \cite{BryantG2}
 \[\Lambda^3_{27}=\{\gamma\in \Lambda^3|\gamma\wedge \phi=0, \gamma\wedge *\phi=0\}.\]
 As an irreducible representation of $G_2$, it has the highest weight $(2,0)$. The space of traceless symmetric bilinear forms on $V$ is also $27$-dimensional and has the highest weight $(2,0)$. So there must be a $G_2$-equivariant isomorphism between them. One such is given in \cite{BryantG2}: Define $\mathbf{j}_\phi:\Lambda^3_{27}\rightarrow {\rm Sym}^2_0(V)$ by
\begin{eqnarray}
 \mathbf{j}_\phi(\gamma)(v,w)=*_\phi\left((v\lrcorner \phi)\wedge (w\lrcorner \phi)\wedge\gamma\right).
\end{eqnarray}
 
\subsection{$G_2$-structures}
Suppose now $M$ is a 7-dimensional smooth manifold. For each $x\in M$, we define the set definite 3-forms 
\[(\Lambda^3_{+})_x=\{\sigma_x\in \Lambda^3T^*_xM|\exists u\in {\rm Hom}_{\mathbf{R}}(T_xM,V), u^*\phi=\sigma_x\}.\]
We call 
\[\Lambda^3_+(M)=\coprod_x(\Lambda^3_+)_x\]
the bundle of definite 3-forms. It is a subbunle of $\Lambda^3(M)$. In fact, $(\Lambda^3_+)_x$ contains a single open orbit of $GL(T_xM)$. Thus the bundle $\Lambda^3_+(M)$ is open in $\Lambda^3$.
\begin{definition}
 A section $\sigma$ of the bundle $\Lambda^3_+(M)$ is called a $G_2${\emph -structure} on $M$.
\end{definition}

A $G_2$-structure on $M$ reduces the total coframe bundle $\mathcal{F}$ to a principal $G_2$-subbundle 
\[F_\sigma=\{(x,u)|u\in {\rm Hom}_{\mathbf{R}}(T_xM,V), u^*\phi=\sigma_x\}.\] 

Because $G_2$ acts reducibly on $\Lambda^pV^*$ for $2\leq p\leq 5$, the vectors bundles of forms $\Lambda^p(M)$ on $M$ decomposes as direct sums correspondingly. For instance
   \[\Lambda^2(M)=\Lambda^2_7(M)\oplus\Lambda^2_{14}(M)\]
for $\Lambda^2_7(M)=F_\sigma\times_{G_2}\Lambda^2_7V^*$ and for $\Lambda^2_{14}(M)=F_\sigma\times_{G_2}\Lambda^2_{14}V^*$.
\subsubsection{Connections and torsions}
  We discuss various connections determined by a $G_2$-structure $\sigma$. 

  First, associated with $\sigma$, there is a canonical metric, denoted by $g^\sigma$. On each tangent space $T_xM$, $g^\sigma$ is simply the pull-back by $u$ of the standard metric on $V$ for any $u\in F_\sigma|_x$. This does not depend on the choice of $u$
 because $u$ only differs by an element of $G_2$. 

 The oriented orthonormal coframe bundle of $g^\sigma$ is obtained from $F_\sigma$ by extension of the structure group to $SO(7)$, i.e., $F_\sigma\cdot SO(7)\subset \mathcal{F}$. 
 
We denote by $\omega=(\omega_{ij})$ the Levi-Civita connection on $F_\sigma\cdot SO(7)$. It is an $\mathfrak{so}(7)$-valued one-form. If $\{\omega_i\}$ is a local orthonormal coframe, i.e., a local section of $F_\sigma\cdot SO(7)$, we have 
  \[d\omega_i=-\omega_{ij}\wedge \omega_j.\]
The dual frame field $e_i$ has covariant derivatives 
  \[\nabla^\sigma_X e_i=e_j\omega_{ji}(X).\]

When we restrict $\omega$ to the subbundle $F_\sigma$, it decomposes as 
  \[\omega=\theta\oplus \tau\]
where $\theta$ takes value in the Lie algebra $\mathfrak{g}_2\subset \mathfrak{so}(7)$ and $\tau$ takes value in $\mathfrak{so}(7)/\mathfrak{g}_2$. It is easy to show that $\theta$ defines a $G_2$-connection on $F_\sigma$. It follow that $\tau$ is semibasic with respect to $F_\sigma\rightarrow M$. Thus $\tau$ is a section of $F_\sigma\times_{G_2}(\mathfrak{so}(7)/\mathfrak{g}_2\otimes V^*)$.
\begin{definition}
 We call $\tau$ the {\emph torsion} tensor of the connection $\theta$.
\end{definition}
For this reason, we call $V_\tau=\mathfrak{so}(7)/\mathfrak{g}_2\otimes V^*$ the torsion space of $G_2$. As an representation of $G_2$, $V_\tau\simeq V^*\otimes V^*$ and thus decomposes into 
\[V_\tau=V_{(0,0)}\oplus V_{(1,0)}\oplus V_{(0,1)}\oplus V_{(2,0)}\] 
where $V_{(p,q)}$ denotes the irreducible representation with the highest weight $(p,q)$, in particular $V_{(1,0)}=V$ and $V_{(0,1)}=\mathfrak{g}_2$. 

Correspondingly $\tau$ has four components, $\tau_0$, $\tau_1$, $\tau_2$, and $\tau_3$. These components show up in the differentials of $\sigma$ and $d*\sigma$:
\begin{proposition}[\cite{BryantG2}]\label{Torsion}
For any $G_2$-structure, there exist unique differential forms $\tau_0\in \Omega^0(M)$, $\tau_1\in \Omega^1(M)$, $\tau_2\in \Omega^2_{14}(M)$, and $\tau_3\in \Omega^3_{27}(M)$ so that the following equations hold:
 \begin{eqnarray}
  d\sigma=\tau_0*\sigma+3\tau_1\wedge\sigma+*\tau_3,\\
  d*\sigma=4\tau_1\wedge*\sigma+\tau_2\wedge\sigma.  
 \end{eqnarray}
\end{proposition}
Various $G_2$ structures with one or several of the components vanishing received much interest in the literature. We are mainly interested in {\emph closed} $G_2$-structures.
\begin{definition}
 A $G_2$-structure is {\emph closed} if its definite 3-form $\sigma$ is closed.
\end{definition}

By Proposition \ref{Torsion}, the condition $d\sigma=0$ is equivalent to $\tau_0=\tau_1=\tau_3=0$. Thus the only torsion component is a 2-form $\tau_2\in \Omega^2_{14}$. 
\begin{definition}
 A $G_2$-structure is called {\emph 1-flat} if $d\sigma=d*\sigma=0$.
\end{definition}

 Equivalently, $\sigma$ is {\emph 1-flat} if and only if $\Delta_\sigma\sigma=0$ where $\Delta_\sigma$ is the Hodge Laplacian of $g^\sigma$. In this case, the connections $\omega$ and $\theta$ coincide and the holonomy of the Levi-Civita connection lies in $G_2$. 
 \subsection{Differential identities}
By composing the exterior differential operator $d$ with various algebraic bundle isomorphisms, we can define many differential operators $d^p_q:\Omega_p\rightarrow \Omega_q$ for all $p,q\in \{1,7,14,27\}$. More explicitly, for $f\in \Omega^0$, $\alpha\in \Omega^1_7$, $\beta\in \Omega^2_{14}$ and $\gamma\in \Omega^3_{27}$,  
\begin{eqnarray*}
d^1_7f=df\\
d^7_7\alpha=*d(\alpha\wedge*\sigma)\\
d^7_{14}\alpha=\pi^2_{14}d\alpha\\
d^7_{27}\alpha=\pi^3_{27}d*(\alpha\wedge*\sigma)\\
d^{14}_{27}\beta=\pi^{3}_{27}d\beta\\
d^{27}_{27}\gamma=*(\pi^4_{27}d\gamma)
\end{eqnarray*}
where $\pi^k_p$ denote the projection onto $\Lambda^k_p$.
and the rest are determined by 
\[(d^p_q)^*=d^q_p\] for $p\neq q$
as well as
 \[d^1_{14}=d^1_{27}=d^1_1=d^{14}_{14}=0.\]
\subsubsection{Torsion free case}
When the $G_2$ structure is torsion free, the Levi-Civita connection $\nabla^\sigma$ coincides with the $G_2$ connetion and its holonomy lies in $G_2$. As in the case of K\"ahler geometry, there are various differential identities by abstract nonsense. These identities are listed in \cite{BryantG2}, pp. 25 (except that the minus sign in the formula $d(\alpha\wedge*\sigma)=-*d^7_7\alpha$ is a typo). These identities are important for our computations. The proof is based on abstract nonsense and constant checking. Besides the identities listed there, We need one more that we now describe. 

Recall that $\Omega^3_{27}$ consists of 3-forms whose wedge products with $\sigma$ and $*\sigma$ vanish. It maybe identified with ${\rm Sym}^2_0$--the space of trace free symmetric bilinear forms. An explicit map is given in \cite{BryantG2} by 
\begin{eqnarray}
 \mathbf{j}_\sigma(\gamma)(\phi)(v,w)=*_\sigma((v\lrcorner \sigma)\wedge(w\lrcorner \sigma)\wedge\gamma)
\end{eqnarray}
for any $\gamma\in \Omega^3_{27}$, $v\in TM$ and $w\in TM$. By definition $d^{27}_{7}\gamma$
takes value in $\Omega^1_7$. So is ${\rm div}(\mathbf{j}_\sigma(\gamma))$. Here we define 
  \[{\rm div}(h)=g^{jl}g^{ik}h_{ij,k}e_l\]
for any symmetric two tensor $h_{ij}$. They must agree up to a constant.
\begin{lemma}\label{DivId}
 If the $G_2$ structure is torsion free, we have
 \begin{eqnarray}
  (d^{27}_7 (\gamma))_\#= A \cdot{\rm div}(\mathbf{j}_\sigma(\gamma))
 \end{eqnarray}
for some universal nonzero constant $A$. 
\end{lemma}
\begin{proof}
 Consider the following diagram
\[\begin{array}{lllll}
\Lambda^3_{27}\otimes \Lambda^1&\stackrel{\delta}{\longrightarrow}&\Lambda^4&\stackrel{\pi^4_{7}}{\longrightarrow}&\Lambda^4_{7}\\
\downarrow \mathbf{j}_\sigma\otimes 1&       &        &                               &  \\ 

{\rm Sym}^2_0\otimes \Lambda^1&\stackrel{c}{\longrightarrow}&\Lambda^1&\stackrel{\wedge\sigma}{\longrightarrow}&\Lambda^4_7
\end{array}\]
where $\delta$ denotes the skewsymmetrization, $c$ denotes the contraction using the metric and $\pi^4_7$ is the orthogonal projection. All spaces have a natural $G_2$ action and all linear maps are $G_2$-equivariant. By Schur's Lemma, there must be a nonzero constant $A,$ so that 
 \[\pi^4_7\circ\delta (\gamma\otimes\alpha)=A'\cdot c(\mathbf{j}_\sigma(\gamma)\otimes\alpha)\wedge\sigma.\]

On the other hand, we have \[d(\gamma)=\delta\circ \nabla^\sigma\gamma,\]
      \[\frac{1}{4}d^{27}_7\gamma\wedge\sigma=\pi^4_7 d(\gamma)=\pi^4_7\circ\delta\circ\nabla^\sigma\gamma\] and 
     \[{\rm div}(\mathbf{j}_\sigma(\gamma))^\flat=c(\nabla^\sigma(\mathbf{j}_\sigma(\gamma)))=c(\mathbf{j}_\sigma\otimes 1(\nabla^\sigma\gamma)).\]
since $\nabla^\sigma$ coincides with the $G_2$ connection and thus $\mathbf{j}_\sigma$ is $\nabla^\sigma$-parallel. 
Combining all these we get
\[d_7^{27}\gamma=A\cdot {\rm div}(\mathbf{j}_\sigma(\gamma))^\flat\]
for $A=4A'$.
\end{proof}
\begin{remark}[How to determine $A$]
 Because it is universal, we might simply assume $M=\mathbf{R}^7$. Since its value will not be important for us, we do not pursue further. %%//Take $\gamma=dx^{267}+dx^{245}-dx^{157}+dx^{146}$ and $\alpha=dx^1$. Then \delta(\gamma\otimes\alpha)=-dx^{12}\wedge(dx^{45}+dx^{67}),\]
 %%                   \[\pi\circ \delta(\gamma\otimes\alpha)=\frac{1}{2}dx^2\wedge\sigma.\] 
\end{remark}
\subsubsection {Torsion case}\label{TorsionId}
 When the underlying $G_2$ structure has torsion, the identity in Lemma \ref{DivId} and the differential identities in \cite{BryantG2} are no longer true in general. One needs to modify the identities by lower order derivatives. However, since we will only be interested in the principal symbols of various differential operators, these lower order terms are not essential. Thus when we do computations, we may proceed as if the underlying structure were torsion free. This observation is very important. Hopefully this rule will be clear once we do concrete calculations.
\subsection{The Laplacian flow}
The flow was introduced in \cite{BryantG2}. The idea is to evolve the definite 3-form $\sigma$ in the direction of its Hodge Laplacian:
 \begin{eqnarray}
  \frac{d}{dt}\sigma=\Delta_\sigma\sigma.
 \end{eqnarray}
Note that, since $\Delta$ depends on $\sigma$ itself, this flow is nonlinear. It has a large symmetry group: In fact, it is preserved by diffeomorphisms. Thus, it is not parabolic. 

The stable solutions are, of course, given by $1-flat$ $G_2$-structures. Then, as is well-known, the underlying metric is Ricci-flat. In this sense, the flow resembles the Ricci flow or the K\"ahler Ricci flow.

Suppose now $\sigma(t)$ is a solution. If the initial structure $\sigma_0$ is closed, then it is not hard to show that a solution $\sigma(t)$ remains closed. Then the flow equation satisfied by $\sigma(t)$ simplifies to 
 \begin{eqnarray}
  \frac{d}{dt}=-d*d*\sigma.
 \end{eqnarray}
Thus $\sigma(t)$ stays in the same cohomology class $[\sigma_0]$. 
\subsection{Hitchin's functional}
Suppose $\sigma_0$ is a closed $G_2$-structure on compact $M^7$. Define 
  \[[\sigma_0]_+=\{\sigma_0+d\beta|\beta\in \Omega^2(M)\}.\]
In \cite{Hitchin}, Hitchin defined a volume functional on $[\sigma_0]$:
 \begin{eqnarray}
  V(\sigma)=\int_M\sigma\wedge*\sigma.
 \end{eqnarray}
He computed the first variation of this functional 
 \begin{eqnarray}\label{VarV}
  \delta V(\delta\sigma)=\frac{7}{18}\int_M \delta \sigma\wedge*\sigma,
 \end{eqnarray}
where $\delta\sigma$ is an exact 3-form. 

The Laplacian flow may be viewed as a gradient flow of $V$ with respect to an unusual metric on $[\sigma_0]$. We now describe this metric. First, note that the tangent space of $[\sigma_0]$ at any point $\sigma\in [\sigma_0]$ is simply $d\Omega^2(M)$. Let $G_\sigma$ be the Green's operator of the Hodge Laplacian $\Delta_\sigma$. We define for any $\psi,\psi'\in d\Omega^2(M)$,
  \begin{eqnarray}\label{WeirdMetric}
   \langle\psi,\psi' \rangle_\sigma=\frac{7}{18}\int_M g^{\sigma}(G_\sigma\psi,\psi')*_\sigma 1.
  \end{eqnarray}
It can be shown that (\ref{WeirdMetric}) in deed defines a metric on $d\Omega^2(M)$.   

Also note $\Delta_\sigma G_\sigma\psi=\psi$ for any exact form $\psi$. With this in mind, we now rewrite (\ref{VarV}) as
 \begin{eqnarray*}
  \begin{array}{lcl}
    \delta V(\delta\sigma)&=&\frac{7}{18}\int_M \delta \sigma\wedge*\sigma\\\\
                          &=&\frac{7}{18}\int_M g^\sigma(\delta\sigma,\sigma) *_\sigma 1\\\\
                          &=&\frac{7}{18}\int_Mg^\sigma(\Delta_\sigma G_\sigma\delta\sigma,\sigma) *_\sigma 1\\\\
                          &=&\frac{7}{18}\int_Mg^\sigma(G_\sigma\delta\sigma,\Delta_\sigma\sigma) *_\sigma 1\\\\
                          &=&\langle \delta_\sigma, \Delta_\sigma\sigma\rangle_\sigma.
\end{array}
 \end{eqnarray*}
Thus under the metric (\ref{WeirdMetric}), the gradient vector field of $V$ is $\Delta_\sigma\sigma$.

\section{Deturck's Trick}
In this section, we apply Deturck's Trick. We have to find the vector field in $\mathcal{L}_{V(\sigma)}\sigma$ so that (\ref{ModFlowtheta}) is strictly parabolic in the direction of closed forms. The vector field is constructed in a similar way to \cite{De}.
 \subsection{Deformation of $G_2$-structures}
The result is due to D. Joyce \cite{Joyce}.
\begin{proposition}
 Let $\sigma_s$ be a family of $G_2$ structures on $M$. Then there exist three differential forms $f^0\in \Omega^0$, $f^1\in \Omega^1$ and $f^3\in \Omega^3_{27}$ uniquely characterized by 
\begin{eqnarray}
 \frac{\partial}{\partial s}\sigma_s=3f^0\sigma+*_\sigma(f^1\wedge\sigma)+f^3.
\end{eqnarray}
 In terms of these, the variation of the metric $g^\sigma$ is given by 
  \begin{eqnarray}
   \frac{\partial}{\partial s}g^\sigma=2f^0g^\sigma+\frac{1}{2}\mathbf{j}_\sigma(f^3).
  \end{eqnarray}
The variation of the dual 4-form $*_\sigma\sigma$ is
\begin{eqnarray}
 \frac{\partial}{\partial s}(*_\sigma\sigma)=4f^0*_\sigma\sigma+f^1\wedge\sigma-*_\sigma f^3.
\end{eqnarray}
\end{proposition}
The proposition will be useful in linearization. If we deform the structure in a closed direction, we have 
\[d(3f^0\sigma+*_\sigma(f^1\wedge\sigma)+f^3)=0.\] 
Using differential identities to expand out we get
\begin{eqnarray}\label{ClosedDeform1}
     0=\frac{4}{7}d^7_1f^1+{\rm l.o.t}\\
 \label{ClosedDeform2} 0=d^1_7f^0+\frac{1}{6}d^7_7f^1+\frac{1}{12}d^{27}_7f^3+{\rm l.o.t}\\
 \label{ClosedDeform3} 0=d^7_{27}f^1+d^{27}_{27}f^3+{\rm l.o.t}.
\end{eqnarray}
\begin{remark}
 The lower order terms are algebraic in $(f^0,f^1,f^3)$, acted on by the various torsion tensors of $\sigma$. It is possible to work them out, but we will not need them in this paper.
\end{remark}

\subsection{The vector field $V(\sigma)$}
Let $\nabla^0$ be a fixed torsion free connection on $M$. Let $\nabla^\sigma$ be the Levi-Civita connection of the Riemannian metric $g^\sigma$ determined by $\sigma$. Then, the difference 
\[T=\nabla^\sigma-\nabla^0\]
is a well-defined tensor on $M$. In fact, since both connections are torsion free, $T$ takes valued in $TM\otimes {\rm Sym}^2T^*M$. Using $g^\sigma$, we identify $TM$ with $T^*M$ and decompose \[TM\otimes {\rm Sym}^2T^*M\simeq TM\oplus TM\otimes {\rm Sym}_0^2T^*M.\]

This gives us two vector fields from $T$: One from $TM$ component and  the other by taking contraction of $TM\otimes {\rm Sym}_0^2T^*M$. To avoid confusing constants, we define $V_1$ and $V_2$ in a slightly differently way. Locally, if $\{e_i\}$ is a frame field and $\{\omega^i\}$ is the dual coframe field and $T=\frac{1}{2}T^i_{jk}e_i\otimes\omega^j\circ \omega^k$ with $T^i_{jk}=T^i_{kj}$, then 
   \[V_1=\frac{1}{7}g^{pq}T^i_{pq}e_i\]
and
   \[V_2=\frac{2}{A}(g^{kj}T^i_{ik}e_j+5V_1)\]
where the repeated indices represent a summation and $A$ is the constant defined in Lemma \ref{DivId}. 

For each pair $(\lambda,\mu)$ of real numbers (that we will determine in the future), define 
 \begin{eqnarray}\label{VectorField}
  V(\sigma)=\lambda V_1+\mu V_2.
 \end{eqnarray}
and correspondingly 
  \begin{eqnarray}
   \mathcal{Q}_{\lambda,\mu}(\sigma)=\lambda\mathcal{L}_{V_1}\sigma+\mu\mathcal{L}_{V_2}\sigma
  \end{eqnarray}
where $\mathcal{L}$ is the Lie derivative. Note that since the vector fields $V_1$ and $V_2$ involves one derivatives of $\sigma$, $\mathcal{Q}_{\lambda,\mu}$ is a second order differential operator on $\sigma$ when $\lambda, \mu$ are not simultaneously 0. 
\subsection{Linearization of $\mathcal{Q}$}
Now we hope to linearize $\mathcal{Q}_{\lambda,\mu}$. By Cartan's formula,
 \[\mathcal{Q}_{\lambda,\mu}\sigma=\lambda d(V_1\lrcorner \sigma)+\lambda V_1\lrcorner d\sigma+\mu d(V_2\lrcorner \sigma)+\mu V_2\lrcorner d\sigma.\]
 If $\sigma$ is closed, 
 \[\mathcal{Q}_{\lambda,\mu}\sigma=d(\mathfrak{q}_{\lambda,\mu}\sigma)\]
with \[\mathfrak{q}_{\lambda,\mu}\sigma=\lambda V_1\lrcorner \sigma+\mu V_2\lrcorner \sigma.\]

Suppose that the variation of $\sigma$ is
  \[\psi=\frac{\partial}{\partial s}|_{s=0}\sigma(s)=3f^0\sigma+*_\sigma(f^1\wedge\sigma)+f^3\] for $f^0\in \Omega^0$, $f^1\in \Omega^1$ and $f^3\in \Omega^3_{27}(M)\subset \Omega^3(M)$. The linearization of $\mathcal{Q}_{\lambda, \mu}$ at $\sigma$ is a second order linear operator $Q_{\lambda,\mu}$ acting on $\psi$ 
\begin{eqnarray}
 Q_{\lambda,\mu}(\psi)=\lambda d(V_{1*}(\psi)\lrcorner \sigma)+\mu d(V_{2*}(\psi)\lrcorner \sigma)+{\rm l.o.t}
\end{eqnarray}
 where $V_{1*}(\psi)=\frac{\partial}{\partial s}V_1(\sigma(s))$ and $V_{2*}(\psi)$ is similarly defined and l.o.t denotes terms of lower order derivatives of $\psi$. Thus the key issue in linearization, as far as the leading term is concerned, is to linearize $V_1$ and $V_2$. 

Consider first the tensor $T$. Since $\nabla^0$ is fixed, the linearization of $T$ is the same as linearizing the connection $\nabla^\sigma$, which, in turn, depends only on the variation $h$ of the underlying metric $g^\sigma$. In fact, as is well known, the Levi-Civita connection is uniquely determined by the torsion free condition and 
\[X(g^\sigma(Y,Z))=g^\sigma(\nabla^\sigma_X Y,Z)+g^\sigma(Y,\nabla^\sigma_XZ).\]
Take the derivative with respect to $t$ to get
\[(\nabla^\sigma_X h)(Y,Z)=g^\sigma(T_*(h)(X,Y),Z)+g^\sigma(Y,T_*(h)(X,Z)).\]
where $T_*(h)(X,Y)=\frac{d}{dt}|_{t=0}\nabla^\sigma_XY$.
Permutate $X,Y,Z$ to get three identities. Add two of them up and subtract from the third to get
\begin{eqnarray}
\begin{array}{l} 
g^\sigma\left(T_*(h)(X,Y),Z\right)\\\\
 =\frac{1}{2}\left[(\nabla^\sigma_Xh)(Y,Z)+(\nabla^\sigma_Yh)(X,Z)-(\nabla^\sigma_Zh)(X,Y)\right].\end{array}
\end{eqnarray}
 Written locally, the variation of $T$ is given by 
\begin{eqnarray}
 (T_*(h))^i_{jk}=\frac{1}{2}g^{il}(h_{jk,l}+h_{kl,j}-h_{lj,k})
\end{eqnarray}
where $h_{ij,k}=(\nabla^\sigma_{e_k}h)_{ij}$.
Thus 
\begin{eqnarray}\label{LinV1}
 V_{1*}(h)=\frac{1}{2}g^{il}g^{jk}h_{jk,l}e_i+{\rm l.o.t}
\end{eqnarray}
and 
\begin{eqnarray}\label{LinV2}
\frac{A}{2}\cdot V_{2*}(h)=g^{ik}g^{pq}h_{kp,q}e_i+(-\frac{1}{2}+\frac{5}{14})g^{ik}g^{pq}h_{pq,k}e_i+{\rm l.o.t}
\end{eqnarray}
where again, {\rm l.o.t} denotes lower order derivatives of $h$. 

Of course, these computations are not new and actually valid for all dimensions. Let us now specify to the $G_2$ case. By Joyce's result, the metric $g$ varies infinitesimally by
\begin{eqnarray}
 h=2f^0g+\frac{1}{2}\mathbf{j}_\sigma(f^3).
\end{eqnarray}
Then the variation of vector fields $V_1$ and $V_2$ are given respectively by
\begin{eqnarray}
 V_{1*}(\psi)=g^{ij}f^0_{,j}e_i+{\rm l.o.t}=(d^1_7f^0)_\#+{\rm l.o.t}.
\end{eqnarray}
and 
\begin{eqnarray}
\begin{array}{ll}
 V_{2*}(\psi)&=\frac{1}{A}(\mathbf{j}_\sigma (f^3))_{ij,k}g^{ik}g^{jl}e_l+{\rm l.o.t}\\\\
             &=(d^{27}_7f^3)_\#+{\rm l.o.t}.
\end{array}
\end{eqnarray}
The last equality is due to Lemma \ref{DivId} and \textsection \ref{TorsionId}.
In conclusion we get the leading term of the linearized operator $Q$
\begin{proposition}
The linearized operator $\mathfrak{q}_{\lambda,\mu}$ is 
\begin{eqnarray}
q_{\lambda,\mu}(\psi)=\lambda (d^1_7f^0)_\#\lrcorner \sigma+\mu (d^{27}_7f^3)_\#\lrcorner \sigma +{\rm l.o.t}.
\end{eqnarray}
 The linearized operator $Q_{\lambda,\mu}$ is 
 \begin{eqnarray}
  Q_{\lambda,\mu}(\psi)=d( q_{\lambda,\mu}(\psi))+{\rm l.o.t}
 \end{eqnarray}
for $\psi=3f^0\sigma+*_\sigma(f^1\wedge\sigma)+f^3$.
\end{proposition}

Now assume further $d\psi=0$, i.e., we are deforming in the direction of closed forms. Using (\ref{ClosedDeform2}), we may rewrite $q$ as 
\begin{eqnarray}\label{Linearq}
      q_{\lambda,\mu}(\psi)=(\lambda-12\mu)*(d^1_7f^0\wedge *\sigma)-2\mu *(d^7_7f^1\wedge*\sigma)+{\rm l.o.t}.
 \end{eqnarray}
\subsection{Linearization of $\mathcal{P}(\sigma)=\Delta_\sigma\sigma$}
Remember the Hodge Laplacian acting on 3-forms is given by
  \[\Delta=*d*d-d*d*.\]
The linearization of $\Delta_\sigma\sigma$ at $\sigma$ is a second order operator on the variation $\psi$. We denote this by $P(\psi)$. For our purpose, we only compute the leading part of $P$: By Joyce's result,
 \[\begin{array}{lcl}
    P(\psi)&=&*d*d(3f^0\sigma+*_\sigma(f^1\wedge\sigma)+f^3)\\\\
           &&-d*d(4f^0*\sigma+f^1\wedge\sigma-*f^3)+{\rm l.o.t}\\\\
           &=& (*d*d-d*d*)f^3\\\\
         &&+(3*d*d+4d*d*)(f^0\sigma)\\\\
          &&+(*d*d+d*d*)*(f^1\wedge\sigma)\\\\
          &&+{\rm l.o.t}.
   \end{array}
\]
It is clear from this that $\mathcal{P}$ is not elliptic. 

Now assume further $d\psi=0$ and $d\sigma=0$. Then we may write 

 \begin{eqnarray*}
  P(\psi)=dp(\psi)
\end{eqnarray*}
with
\begin{eqnarray}
p(\psi)=-*d(4f^0*\sigma+f^1\wedge\sigma-*f^3)+{\rm l.o.t.}
\end{eqnarray}
 \subsection{Parabolicity along closed forms}
 Now we choose parameters $\lambda=-5,\mu=-1$. We use the formula (\cite{BryantG2}, pp.25) 
 \[df^1=\frac{1}{3}*(d^7_7f^1\wedge*\sigma)+d^7_{14}f^1\]
and \[d(f^1\wedge\sigma)=\frac{2}{3}d^7_7f^1\wedge*\sigma-*d^7_{14}f^1\]
 to compute that 
 \begin{eqnarray*}p(\psi)+q(\psi)&=&3*(df^0\wedge*\sigma)+*(d^7_7f^1\wedge*\sigma)+*d*f^3+{\rm l.o.t.}\\
    &=&3*(df^0\wedge*\sigma)+*d*f^3+\frac{4}{3}*(d^7_7f^1\wedge*\sigma)+d^7_{14}f^1+{\rm l.o.t.}
    \end{eqnarray*}  
 
 On the other hand 
 \begin{eqnarray*}
 *d*\psi&=&*d*(3f^0\sigma+*(f^1\wedge\sigma)+f^3)\\
             &=&3*(df^0\wedge*\sigma)+\frac{2}{3}*(d^7_7f^1\wedge*\sigma)-d^7_{14}f^1+*d*f^3+{\rm l.o.t}.
\end{eqnarray*}
 Hence
\begin{eqnarray*}
p(\psi)+q(\psi)-*d*\psi=\frac{2}{3}*(d^7_7f^1\wedge*\sigma)+2d^7_{14}f^1+{\rm l.o.t.}=df^1+{\rm l.o.t.}
\end{eqnarray*}
Now take the exterior differential of this identity and use the fact 
\[-\Delta_\sigma\psi=d*d*\psi\] when $d\psi=0$. We get:
\begin{lemma}\label{Parabolicity}
If $d\psi=0$, then $P(\psi)+Q_{-5,-1}(\psi)=-\Delta_\sigma\psi+d \Phi(\psi)$ where $\Phi(\psi)$ is an algebraic linear operator on $\psi$ with coefficients depending on the torsion of $\sigma$ in a universal way.   
\end{lemma}

\section{Proof of Theorem \ref{Thm} assuming Lemma \ref{Lemma}}\label{Proof}
 The proof is modeled on Ricci flow except that, for the uniqueness, we will avoid the notion of harmonic map flow. 
\subsection{Short time existence}
For the first step, we describe a clever device used in both \cite{Hamilton} and \cite{Hamilton1}.  Let $\sigma_0$ be an initial closed $G_2$-structure. We solve formally for the derivative $\frac{d^k}{dt^k}\theta(0)$ by differentiating through the equation (\ref{Flowtheta}). Let $\hat{\theta}(t)$ be a family of 3-forms whose Taylor series expansion at $t=0$ is given by $\frac{d^k}{dt^k}\theta(0)$ (the existence is a non-trivial real analysis exercise) . Let 
\[\chi(t)=\frac{d}{dt}\hat{\theta}-\Delta_{\hat{\theta}+\sigma_0}(\hat{\theta}+\sigma_0)-\mathcal{L}_{V(\hat{\theta}+\sigma_0)}(\hat{\theta}+\sigma_0)\]
where $V$ is defined in (\ref{VectorField}). 
Then $\chi(t)$ has a zero Taylor series expansion at $t=0$. By translating $\chi$ in the $t$ direction, we obtain a 3-form $\overline{\chi}(t)$ as close to $\chi$ as possible that vanishes in $[0,\epsilon']$ for some $\epsilon'>0$. By applying Lemma (\ref{Lemma}) to the pair $(\overline{\chi},0)$, we get a solution $\overline{\theta}(t)$ to (\ref{ModFlowtheta}) for a possibly shorter time period $[0,\epsilon]$.  As before, denote $\overline{\sigma}(t)=\theta(t)+\sigma_0$.

Now let the $\phi_t$ be time-dependent diffeomorphisms with vector fields $-V(\overline{\sigma})$. In other word, $\phi_t$ is the solution to the ODE
 \begin{eqnarray*}
\left\{\begin{array}{l}
\frac{d}{dt}\phi_t(x)=-V(\overline{\sigma}(t))|_{\phi(x)}\\\\
  \phi_0=Id\end{array}\right..
 \end{eqnarray*}

Let 
  \[\sigma(t)=\phi_t^*\overline{\sigma}(t).\]
Then \[\sigma(0)=\sigma_0\] since $\phi_0=Id$.
Also,
     \[d\sigma(t)=\phi_t^*d(\overline{\sigma}(t))=0\] since $\overline{\sigma}(t)$ is closed.
  Moreover,
     \begin{eqnarray*}
\begin{array}{ll}
     \frac{d}{dt}\sigma&=\phi_t^*(\mathcal{L}_{-V}\overline{\sigma})+\phi_t^*(\Delta_{\overline{\sigma}}\overline{\sigma})+\phi_t^*(\mathcal{L}_V\overline{\sigma})\\\\
                      &= \Delta_{\phi_t^*\overline{\sigma}}\phi^*_t\overline{\sigma}\\\\
                      &=\Delta_\sigma\sigma.
\end{array}
     \end{eqnarray*}
Thus $\sigma(t)$ solves the flow equation (\ref{ClsdLapFlow}). Theorem \ref{Thm} follows.

\subsection{Uniqueness}
 Given a family of $G_2$-structures $\sigma_t$, consider the following nonlinear evolution equation for diffeomorphisms
 \begin{eqnarray}\label{DiffFlow}\frac{d}{dt}\phi=-V((\phi^{-1})^*\sigma)\end{eqnarray}
where $V=\lambda V_1+\mu V_2$ with $\lambda=-5,\mu=-1$ as usual. We claim:
\begin{proposition}
 The flow (\ref{DiffFlow}) is strictly parabolic in $\phi$.
\end{proposition}
\begin{proof}
 We linearize (\ref{DiffFlow}). Suppose that \[\frac{\partial}{\partial s}|_{s=0}\phi=X.\] Then it may be computed that 
  \[\frac{\partial}{\partial s}|_{s=0}\phi^{-1}=-(\phi^{-1})_*X,\]
 and that 
\[\begin{array}{lcl}\frac{\partial}{\partial s}|_{s=0}(\phi^{-1})^*\sigma&=&(\phi^{-1})^{*}\mathcal{L}_{-(\phi^{-1})_*X}\sigma\\\\
                            &=&-\mathcal{L}_X(\phi^{-1})^*\sigma. 
  \end{array}
 \]
Denote $\overline{\sigma}=(\phi^{-1})^*\sigma$, which is also a $G_2$-structure. We use (\ref{LinV1})  and (\ref{LinV2}) to compute the linearization of the right hand side of (\ref{DiffFlow}):
\[\begin{array}{lcl}
   -\frac{\partial}{\partial s}|_{s=0}V((\phi^{-1})^*\sigma)&=&\lambda (V_1)_*(\mathcal{L}_X\overline{\sigma})+\mu (V_2)_*(\mathcal{L}_X\overline{\sigma})\\\\
                           &=& \lambda (V_1)_*(d(X\lrcorner \overline{\sigma}))+\mu (V_2)_*(d(X\lrcorner \overline{\sigma}))+{\rm l.o.t}\\\\
                           &=&-5 (-\frac{1}{7}d^1_7d^7_1X^\flat)-1(d^{27}_7d^7_{27}X^\flat)+{\rm l.o.t.}\\\\
                           &=&-(d^7_7d^7_7+d^1_7d^7_1)X^\flat+{\rm l.o.t.}\\\\
                           &=&-\Delta_{\overline{\sigma}}X^\flat+{\rm l.o.t.}
  \end{array}
\]
where all operators are defined using $\overline{\sigma}$ and ${\rm l.o.t.}$ denotes terms involving $X$ derivatives of degree less than 2.\end{proof}

Suppose now $\sigma(t)$ solves (\ref{LapFlow}). Let $\phi(t)$ be the diffeomorphisms obtained from solving (\ref{DiffFlow}) with the initial condition $\phi(0)=Id$. Then the new family $\overline{\sigma}=(\phi(t)^{-1})^*\sigma(t)$ satisfies  
\[\begin{array}{lcl}
   \frac{d}{dt}\overline{\sigma}&=&(\phi^{-1})^{*}\mathcal{L}_{((\phi^{-1})_*V(\overline{\sigma}))}\sigma+(\phi^{-1})^*\Delta_\sigma\sigma\\\\
                                &=&\mathcal{L}_{V(\overline{\sigma})}\overline{\sigma}+\Delta_{\overline{\sigma}}\overline{\sigma},
  \end{array}
\]
i.e., the flow equation (\ref{ModFlow}).

Now we are ready to prove uniqueness. Suppose $\sigma_a (a=1,2)$ are two families of closed $G_2$ structures solving (\ref{LapFlow}). Let $\phi_a$ be the corresponding families of diffeomorphisms solving (\ref{DiffFlow}) with the initial condition $\phi_a(0)=Id$. Then we have shown that \[\overline{\sigma}_a=(\phi^{-1})^*\sigma_a\] both solve (\ref{ModFlow}) and have the same initial condition. It is clear that $\overline{\sigma}_a$ are closed. Hence, by uniqueness part of Lemma \ref{Lemma}, \[\overline{\sigma}_1=\overline{\sigma}_2:=\overline{\sigma}.\]
But then $\phi_a$ solve the same ODE 
 \[\frac{d}{dt}\phi=V(\overline{\sigma})\] with the initial condition $\phi(0)=Id$. Therefore they must be the same: $\phi_1=\phi_2$. Hence $\sigma_1=\sigma_2$. 
\section{Proof of Lemma \ref{Lemma}}\label{ProofLemma}
Throughout this section we assume that $M^7$ is compact. 

We will use Nash Moser Inverse Function theorem. Our application of this theorem is rather formal, i.e., it does not require any hard estimates. We state the theorem as follows (\cite{Hamilton1}, pp. 171-172):
\begin{theorem}[Nash Moser Theorem]
 Let $\mathcal{F}$ and $\mathcal{G}$ be tame Fr\'echet spaces and $F:\mathcal{U}\subset \mathcal{F}\rightarrow \mathcal{G}$ a smooth tame map. Suppose that the equation for the derivative $DF(f)h=k$ has a unique solution $f=VF(f)k$ for all $f$ in $\mathcal{U}$ and all $k$ and that the family of inverses $VF:\mathcal{U}\times \mathcal{G}\rightarrow \mathcal{F}$ is a smooth tame map. Then $F$ is locally invertible and each local inverse $F^{-1}$ is a smooth tame map.
\end{theorem}
 
A tame Fr\'echet space $\mathcal{F}$ is endowed with a countable family of increasing norms $\{\parallel \parallel_n\}_{n=1}^\infty$ defining the topology (such a family is called a grading). Thus a sequence $\{x_k\}$ converges if and only if it converges with respect to each norm. A subset $\mathcal{U}\subset \mathcal{F}$ is open if and only if, for each $x\in \mathcal{U}$, a small $\parallel \parallel_n$ ball around $x$ is contained in $\mathcal{U}$ for some $n$. Two such gradings $\parallel\parallel_n$ and $\parallel\parallel'_n$ are said to be tamely equivalent with degree $r$ and base $b$ if $\parallel \parallel_n\leq C(n) \parallel \parallel'_{n+r}$ and $\parallel\parallel'_n \leq C(n)\parallel \parallel_{n+r} $ for all $n>b$.

 A continuous map $F:\mathcal{U}\rightarrow \mathcal{G}$ is called tame if each $x\in \mathcal{U}$ has a neighborhood on which a tame estimate holds for $F$ with base $b$ and degree $r$, i.e., 
  \[\parallel F(y)\parallel _n\leq C (1+\parallel y\parallel_{n+r})\]
for all $n>b$. Here the numbers $b$ and $r$ (called the degree) may vary from neighborhood to neighborhood and $C$ depends on $n$ and the neighborhood under consideration. Such a map is called smooth tame if all its derivatives are tame.

Many interesting smooth tame maps arise from partial differential operators. For our purpose, we consider a class of tame spaces introduced in \cite{Hamilton}. Let $V$ be a vector bundle over $M$. For a section $f$ of $V$, we let $|f|_n$ denote the $L^2$ norm of $f$ and its covariant derivatives up to degree $n$. For a time-dependent section $f$ in $\mathcal{F}=C^\infty([0,T]\times M, V)$, we put
\[|f|_n^2=\int_0^T |f(t)|_n^2dt,\]
so that $|f|_n$ only measures space derivatives. Then let
\[\parallel f\parallel_n=\sum_{2j\leq n}| (\frac{\partial} {\partial t})^jf|_{n-2j}.\]
This is a weighted norm counting one time derivative equal to two space derivatives. These norms make $\mathcal{F}=C^\infty\left([0,T]\times M, V\right)$ a tame space (this fact is not proved in \cite{Hamilton}). 

By Sobolev embedding, a tamely equivalent grading $|[\quad]|_n$ is given by 
\[|[f]|_n=\sum_{2j\leq n}[ (\frac{\partial} {\partial t})^jf]_{n-2j} \]
where $[f]_n$ is the sup-norm of $f$ and its space derivatives upto degree $n$.

If $P:\mathcal{U}\subset C^{\infty}([0,T]\times M, V)\rightarrow C^\infty\left([0,T]\times M, W\right)$ is a smooth (nonlinear) partial differential operator between sections of vector bundles, it is clear that $P$ is a smooth tame map. On the other hand, the set $D^r(V,W)$ of differential operators of degree $\leq r$ is itself a vector bundle and a linear operator of degree $r$ may be viewed as a section of this vector bundle.   

The Fr\'echet spaces we will be considering are $\mathcal{F}=dC^{\infty}\left([0,T]\times M, \Lambda^2(M)\right)$ and $\mathcal{G}=\mathcal{F}\times d C^\infty \left(M,\Lambda^2(M)\right)$. 
\begin{proposition}
 The Fr\'echet space $\mathcal{F}$ is tame with the grading $\parallel\parallel_n$ restricted from $C^\infty\left([0,T]\times M,\Lambda^3(M)\right)$. The Fr\'echet space $\mathcal{G}$ is tame under the grading $\parallel\parallel_n+||_n$.
\end{proposition}
\begin{proof}
 Fix a Riemannian metric on $M$. Let $G$ be the Green's operator of the Hodge Laplacian $\Delta$ on $M$. We extend $G$ to families of forms on $M\times [0,T]$ by constancy in $t$. Then from the standard Hodge theory (or from Theorem 3.3 and Lemma 3.3.4 in \cite{Hamilton1}, pp. 158), we have the following estimates
\[|G(\alpha)|_n\leq C |\alpha|_{n-2},\]
where $C$ depends on $n$, $M$ and the metric. 
It follows from this estimate that
\[\parallel G( \alpha)\parallel_n\leq C\parallel \alpha\parallel_{n-2}.\]  
Thus $G$ is a linear tame map.  

Now for any $\psi \in \mathcal{F}$, we have $\psi=dGd^*(\psi)$. It follows that $\mathcal{F}$ is a closed subspace of $C^{\infty}\left([0,T]\times M, \Lambda^2(M)\right)$. Since $dGd^*$ is a composite of tame maps, it is tame. Thus,  $\mathcal{F}$ is a direct summand of the tame space $C^{\infty}\left([0,T]\times M, \Lambda^2(M)\right)$. The result now follows from Lemma 1.3.3 in \cite{Hamilton1}, pp. 136. 

For the same reason, $d C^\infty \left(M,\Lambda^2(M)\right)$ is tame. The product $\mathcal{G}$ is thus tame.
\end{proof}

The map interesting to us is $F:\mathcal{F}\rightarrow \mathcal{F}\times \mathcal{G}$:

\begin{eqnarray}
 F(\theta)=\left(\frac{d}{dt}\theta-\Delta_\sigma\sigma-\mathcal{L}_{V(\sigma)}\sigma,\theta|_{t=0}\right)
\end{eqnarray}
with $\sigma=\sigma_0+\theta$.
We have computed the derivative of $F$ 
\[F_*(\psi)=\left(\frac{d}{dt}\psi-P(\psi)-Q_{-5,-1}(\psi), \psi|_{t=0}\right).\]
Note that the operators $P$ and $Q$ depends smoothly on $\theta$. In fact, the coefficients are given by universal functions of torsion and curvature of $\sigma=\sigma_0+\theta$.
\subsection{Injectivity of $F_*$}
This is relatively easy based on Lemma \ref{Parabolicity}.
\begin{lemma}\label{Injectivity}
The system 
 \[\frac{d}{dt}\psi-P(\psi)-Q_{-5,-1}(\psi)=0,\]
\[\psi(0)=0,\]
 \[d\psi=0\]
has a unique solution $\psi=0.$
\end{lemma}
\begin{proof}
We have proved that if 
\[d\psi=0,\]
then \[P(\psi)+Q_{-5,-1}(\psi)=-\Delta_\sigma\psi+{\rm l.o.t}.\]
Thus $\psi$ satisfy an evolution equation of the form
\[\frac{d}{dt}\psi=-\Delta_\sigma\psi+{\rm l.o.t.}\] 
 Thus the lemma follows from the standard linear parabolic theory.
\end{proof}
\subsection{Surjectivity of $F_*$}
We have to show that for any time-dependent 3-form $\chi$ and any 3-form $\psi_0$, there is a family of 3-forms $\psi $ such that 
\begin{eqnarray}\label{Flowpsi}
\left\{
\begin{array}{l}
\frac{d}{dt}\psi-P(\psi)-Q_{-5,-1}(\psi)=\chi,\\\\
\psi(0)=\psi_0.
\end{array}\right.
\end{eqnarray}
  
We can write $\chi=d\xi$ and $\psi_0=d\beta_0$ for a time-dependent 2-form $\xi$ and a 2-form $\beta_0$. Rather than solve $\psi$, we may try to solve the following equation for $\beta$:
\begin{eqnarray*}
 \frac{d}{dt}\beta=-\Delta_\sigma\beta+\Phi(d\beta)+\xi,\\
 \beta(0)=\beta_0.
\end{eqnarray*}

The existence of $\beta(t)$ follows from standard parabolic theory. Set $\psi=d\beta$. We get:
\begin{lemma}\label{Surjectivity}
 The map $F_*$ is surjective.
\end{lemma}

\begin{proof}
It is clear that $\psi$ solves
\begin{eqnarray*}
\frac{d}{dt}\psi=-\Delta_\sigma\psi+d\Phi(\psi)+\chi\\
\psi(0)=\psi_0.
\end{eqnarray*}
Because $d\psi=0$, and $d\sigma=0$, by Lemma \ref{Parabolicity}, 
\[P(\psi)+Q_{-5,-1}\psi=-\Delta_\sigma\psi+d\Phi(\psi).\] Hence $\psi$ solves (\ref{Flowpsi}).
\end{proof}

Combining Lemma (\ref{Injectivity}) and Lemma (\ref{Surjectivity}) we conclude that $F_*|_\theta$ is an isomorphism. Thus we may define the family of inverses $F_*^{-1}:\mathcal{U}\times \mathcal{G}\rightarrow \mathcal{F}$. To apply Nash Moser inverse function theorem, we need to prove $F_*^{-1}$ is smooth tame.
\subsection{$F_*^{-1}$ is smooth tame}
Usually, this amounts to certain a priori estimates of the solutions of a family linear equations. In our case, since we have showed that a solution $\psi$ must satisfy a parabolic equation of the form
 \[\frac{d}{dt}\psi=-\Delta_\sigma\psi+d\Phi(\psi),\]
this should follow from standard linear parabolic theory. However, the reader may find it difficult to find relevant results in the literature. Fortunately, R. Hamilton essentially did this work in \cite{Hamilton}. His result is general enough for us to avoid annoying estimates. For completeness, we state this result. 

Let $X$ be a compact manifold and let $V$ and $W$ be vector bundles over $X$. Consider a system of linear evolution equations on $0\leq t\leq T$ for sections $f$ of $V$ and $g$ of $W$
 \[\frac{d}{dt}f=Pf+Lg+h,\quad \frac{d}{dt}g=Mf+Ng+k\]
where $P, L, M, N$ are linear differential operators involving only space derivatives whose coefficients are smooth functions of both space and time. We assume $P$ has degree 2, $L$ and $M$ has degree 1, and $N$ has degree 0. 

First, Hamilton showed that if the first equation is parabolic, then for any $(f_0,g_0,h,k)$, there exists a unique smooth solution $(f,g)$ with $f=f_0$ and $g=g_0$ at $t=0$. Then he proved the following 
\begin{theorem}[\cite{Hamilton}, pp. 266]
 Let the solution $(f,g)$ of the system be written as a function 
  \[(f,g)=S(P,L,M,N,h,k,f_0,g_0)\]
of the coefficients $P,L,,N$ and the data $h,k$ and the initial values $f_0,g_0$. In the open set where $P$ is parabolic, the solution $S$ is a smooth tame map in the gradings $\parallel\parallel_n$ on $f,g,h,k$ and $||_n$ on $f_0,g_0$ and $|[\quad]|_n$ on $P,L,M,N$.
\end{theorem}

A corollary is the the statement for a single parabolic system.
\begin{corollary}
 Let the solution $f$ to the system 
  \[\frac{d}{dt}f=Pf+h\] with the initial condition $f(0)=f_0$
be written as a function 
\[f=S(P,h,f_0).\]
In the open set where $P$ is parabolic, the solution $S$ is a smooth tame map in the gradings $\parallel\parallel_n$ on $f,h$ and $||_n$ on $f_0$ and $|[\quad]|_n$ on $P$.
\end{corollary}
\begin{proof}
 The assignment \[(P,h,f_0)\mapsto (P,L=0,M=0,N=0,h,k=0,f_0,g_0=0)\] is clearly a smooth tame map. The composite of two smooth tame maps is smooth tame.
\end{proof}

Now we use this to prove $F^{-1}_*$ is smooth tame.
\begin{lemma}\label{tame}
 The map $F^{-1}_*:\mathcal{U}\times \mathcal{G}\rightarrow \mathcal{F}$ is smooth tame.
\end{lemma}
  
\begin{proof}
 We have shown that $\psi=F^{-1}_*(\theta,\chi, \psi_0)$ is the unique solution to the parabolic system 
  \[\frac{d}{dt}\psi=-\Delta_{\sigma_0+\theta}\psi+d\Phi(\psi)+\chi\]
with $\psi(0)=\psi_0$.
Thus $\psi$ may be written as 
 \[\psi=S(P(\theta),\chi,\psi_0)\]
where $P(\theta)$ is the assignment to $\theta$ the linear operator $-\Delta_{\sigma_0+\theta}+d\Phi(\psi)$ in above equation (do not confuse $P$ with the linearization of $\Delta_\sigma\sigma$). Now $P$ itself can be viewed as a composite of the inclusion $i:\mathcal{U}\rightarrow C^\infty(M, \Lambda^3_+)$ with the nonlinear partial differential operator $C^\infty(M,\Lambda^3_+)\rightarrow D^2(\Lambda^3,\Lambda^3)$ by assigning $\sigma$ to $-\Delta_\sigma+{\rm l.o.t}$. Thus $P$ is smooth tame (keeping in mind the tame equivalence of the gradings $\parallel\parallel_n$ and $|[\quad]|_n$). 

Hence the map $F^{-1}_*:\mathcal{U}\times \mathcal{G}$ may be factored as 
\[F^{-1}_*=S\circ (P\times i) \]
where $i$ is the inclusion of $\mathcal{G}$ into $C^\infty([0,T]\times M, \Lambda^3(M))\times C^\infty (M,\Lambda^3(M))$. Since all maps involved are smooth tame, so is $F^{-1}_*$. 
\end{proof}

Finally, Lemma \ref{Lemma} follows from Nash Moser inverse function theorem.

\end{document}